\documentclass[12pt]{amsart}

\usepackage{amssymb,amsthm}
\usepackage{eucal}
\usepackage[shortlabels]{enumitem}
\usepackage{thmtools,thm-restate}
\usepackage{hyperref}

\setlist[description]{font=\normalfont}

\usepackage{interval}
\intervalconfig{soft  open  fences}

\usepackage{mathtools}

\declaretheorem{theorem}
\declaretheorem[sibling=theorem]{proposition}
\declaretheorem[sibling=theorem]{lemma}

\declaretheorem[sibling=theorem]{claim}

\declaretheorem[style=definition,sibling=theorem]{definition}

\declaretheorem[style=remark,sibling=theorem]{remark}
\declaretheorem[sibling=theorem]{fact}

\declaretheorem{question}

\DeclareMathOperator{\cof}{cof}

\DeclareMathOperator{\dom}{dom}

\newcommand{\seq}[2]{\langle #1 : #2 \rangle}

\newcommand{\surj}{\twoheadrightarrow}

\newcommand{\Col}{\textup{Col}}
\newcommand{\Add}{\textup{Add}}
\newcommand{\ON}{\textup{ON}}
\newcommand{\AP}{\textup{\textsf{{AP}}}}

\newcommand{\ZFC}{\textup{\textsf{ZFC}}}

\newcommand{\DSS}{\textup{\textsf{DSS}}}

\newcommand{\rest}{\upharpoonright}

\newcommand{\nrest}{\!\rest\!}

\newcommand{\M}{\mathbb M}
\renewcommand{\P}{\mathbb P}

\newcommand{\A}{\mathbb A}
\newcommand{\B}{\mathbb B}

\newcommand{\T}{\mathbb T}

\newcommand{\I}{\mathbb I}

\title{On Disjoint Stationary Sequences}
\author{Maxwell Levine}

\address{Albert-Ludwigs-Universit\"at Freiburg,
Mathematisches Institut, Abteilung f\"ur math. Logik, Ernst--Zermelo--Stra\ss e~1, 
79104 Freiburg im Breisgau, Germany}
\email{maxwell.levine@mathematik.uni-freiburg.de}
\keywords{Forcing, large cardinals, equiconsistency}
\subjclass[2010]{03E35, 03E55}

\begin{document}

\begin{abstract} We answer a question of Krueger by obtaining disjoint stationary sequences on successive cardinals. The main idea is an alternative presentation of a mixed support iteration, using it even more explicitly as a variant of Mitchell forcing. We also use a Mahlo cardinal to obtain a model in which $\aleph_2 \notin I[\aleph_2]$ and there is no disjoint stationary sequence on $\aleph_2$, answering a question of Gilton.\end{abstract}

\maketitle


\section{Introduction and Background}

In order to develop the theory of infinite cardinals, set theorists study a variety of objects that can potentially exist on these cardinals. The objects of interest for this paper are called \emph{disjoint stationary sequences}. These were introduced by Krueger to answer a question of Abraham and Shelah about forcing clubs through stationary sets \cite{Abraham-Shelah1983}. Beginning in joint work with Friedman, Krueger wrote a series of papers in this area, connecting a wide range of concepts and answering seemingly unrelated questions of Foreman and Todor{\v c}evi{\' c} \cite{Friedman-Krueger2007, Krueger2007, Krueger2008b, Krueger2008, Krueger2008a, Krueger2009}. The purpose of this paper is to further develop this area.


Krueger's new arguments generally hinged on the behavior of two-step iterations of the form $\Add(\tau) \ast \P$. In order to extend the application of these arguments as widely as possible, Krueger developed the notion of mixed support forcing \cite{Krueger2008b, Krueger2009}, which had apparently been part of the folklore for some time. These forcings are to some extent an analog of the forcing that Mitchell used to obtain the tree property at double successors of regular cardinals. Their most notable feature is the appearance of quotients insofar as the forcings took the form $\M \simeq \bar{\M} \ast \Add(\tau) \ast \mathbb{E}$ where $\bar \M$ is a partial mixed support iteration. The appearance of $\Add(\tau)$ after the initial component, together with the preservation properties of the quotient $\mathbb{E}$, allowed Krueger's new arguments to go through various complicated constructions. Mixed support iterations have found several applications since \cite{Gilton-Krueger2020}, particularly in regard to guessing models \cite{Viale2012}.


The main idea in this paper is to use a version of Mitchell forcing to accomplish the task of a mixed support iteration. Specifically, we prove that this version of Mitchell forcing takes the form $\M \simeq \bar{\M} \ast \Add(\tau) \ast \mathbb{E}$.\footnote{The extent to which all variations of these forcings are equivalent or not is left as a loose end. Here we only deal with the case where the two-step iteration $\Add(\tau) \ast \P$ takes the form $\Add(\tau) \ast \dot{\Col}(\mu,\delta)$.} The trick used to obtain this structural property goes back to Mitchell's thesis and is also reminiscent of the one used by Cummings et al$.$ in ``The Eightfold Way'' to demonstrate that subtle variations in the definitions of Mitchell forcing---up to merely shifting a L{\'e}vy collapse by a single coordinate---can substantially alter the properties of the forcing extension. The benefit of the forcing used here is that it comes with a projection analysis of the sort that Abraham used for Mitchell forcing \cite{Abraham1983}. Both the forcing itself and its quotients are projections of products of the form $\A \times \T$ where $\A$ has a good chain condition and $\T$ has a good closure property. This allows us to obtain preservation properties conveniently, without having to delve into too many technical details. Abraham in fact used this projection analysis to extend Mitchell's result to successive cardinals. This is exactly what we do here for disjoint stationary sequences, answering the first component of a question of Krueger \cite[Question 12.8]{Krueger2009}:

\begin{theorem}\label{consecutivetheorem} Suppose $\lambda_1<\lambda_2$ are two Mahlo cardinals in $V$. Then there is a forcing extension in which there are disjoint stationary sequences on $\aleph_2$ and $\aleph_3$.\end{theorem}

We lay out the basic definition and concepts in the following subsections and then develop the proof in \autoref{sec-thm1and2}. We also achieve one of Krueger's separations for successive cardinals, which answers a component of another one of his questions \cite[Question 12.9]{Krueger2009}:

\begin{theorem}\label{superconsecutivetheorem} Suppose $\lambda_1<\lambda_2$ are two Mahlo cardinals in $V$. Then there is a forcing extension in which for $\mu \in \{\aleph_1,\aleph_2\}$, there are stationarily many $N \in [H(\mu^+)]^\mu$ that are internally stationary but not internally club.\end{theorem}


%




The last main result is motivated by work of Gilton and Krueger, who answered a question from ``The Eightfold Way'' by obtaining stationary reflection for subsets of $\aleph_2 \cap \cof(\omega)$ together with failure of approachability at $\aleph_2$ (i.e$.$ $\aleph_2 \notin I[\aleph_2]$) using disjoint stationary sequences \cite{Gilton-Krueger2020}. This result used the fact that the existence of a disjoint stationary sequence implies failure of approachability. Gilton asked for the exact consistency strength of the failure of approachability at $\aleph_2$ together with the nonexistence of a disjoint stationary sequence on $\aleph_2$ \cite[Question 9.0.15]{Gilton-thesis}. (He pointed out that Cox found this separation using $\textsf{\textup{PFA}}$ \cite{Cox2021}.) It is known that the failure of approachability requires the consistency strength of a Mahlo cardinal since $\square_\tau$ holds if $\tau^+$ is not Mahlo in $L$ \cite{Jensen1972} and $\square_\tau$ implies the approachability property $\tau^+ \in I[\tau^+]$ \cite{Handbook-Eisworth}. In \autoref{sec-thm3} we show that a Mahlo cardinal is sufficient for the separation:

\begin{theorem}\label{secondarytheorem} Suppose that $\lambda$ is Mahlo in $V$. Then there is a forcing extension in which $\aleph_2 \notin I[\aleph_2]$ and there is no disjoint stationary sequence on $\aleph_2$.\end{theorem}

Disjoint stationary sequences are known to be interpretable in terms of canonical structure (see \autoref{dss-equivalence} below), and the main idea for \autoref{secondarytheorem} is a simple master condition argument that exploits this connection.


\subsection{Basic Definitions}

We assume familiarity with the basics of forcing and large cardinals. We use the following conventions: If $\P$ is a forcing poset, then $p \le q$ for $p,q \in \P$ means that $p$ is stronger than $q$. We say that $\P$ is $\kappa$-closed if for all $\le_\P$-decreasing sequences $\seq{p_\xi}{\xi<\tau}$ with $\tau<\kappa$, there is a lower bound $p$, i.e$.$ $p\le p_\xi$ for all $\xi<\tau$. (Not all authors use this formulation of $\kappa$-closedness.) We say that $\P$ has the $\kappa$-chain condition if all antichains $A \subseteq \P$ have cardinality strictly less than $\kappa$. All posets considered will be separative.

Now we give our main definitions:

\begin{definition} Given a regular cardinal $\mu$, a \emph{disjoint stationary sequence} on $\mu^+$ is a sequence $\seq{\mathcal{S}_\alpha}{\alpha \in S}$ such that:

\begin{itemize}

\item $S \subseteq \mu^+ \cap \cof(\mu)$ is stationary,
\item $\mathcal{S}_\alpha$ is a stationary subset of $P_\mu(\alpha)$ for all $\alpha \in S$,
\item $\mathcal{S}_\alpha \cap \mathcal{S}_\beta = \emptyset$ if $\alpha \ne \beta$.
\end{itemize}


We write $\DSS(\mu^+)$ to say that there is a disjoint stationary sequence on $\mu^+$.
\end{definition}


\begin{definition} Given an uncountable regular $\kappa$ and a set $N \in [H(\Theta)]^\kappa$,\footnote{See Jech for details on stationary sets \cite{Jech2003}.} we say: 

\begin{itemize}

\item $N$ is \emph{internally unbounded} if $\forall x \in P_\kappa(N), \exists M \in N, x \subseteq M$,
\item $N$ is \emph{internally stationary} if $P_\kappa(N) \cap N$ is stationary in $P_\kappa(N)$,
\item $N$ is \emph{internally club} if $P_\kappa(N) \cap N$ is club in $P_\kappa(N)$,
\item $N$ is \emph{internally approachable} if there is an increasing and continuous chain $\seq{M_\xi}{\xi<\kappa}$ such that $|M_\xi|<\kappa$ and $\seq{M_\eta}{\eta<\xi} \in M_{\xi+1}$  for all $\xi<\kappa$ such that $N=\bigcup_{\xi<\kappa} M_\xi$.

\end{itemize}
\end{definition}

Although disjoint stationary sequences may seem unrelated to the separation of variants of internal approachability, there are deep connections here, for example:

\begin{fact}[Krueger, \cite{Krueger2009}]\label{dss-equivalence} If $\mu$ is regular and $2^\mu=\mu^+$, then $\DSS(\mu^+)$ is equivalent to the existence of a stationary set $U \subseteq [H(\mu^+)]^\mu$ such that every $N \in U$ is internally unbounded but not internally club.\end{fact}




\subsection{Projections and Preservation Lemmas}

Technically speaking, our main goal is to show that certain forcing quotients behave nicely. We will make an effort to demonstrate the preservation properties of these quotients directly. These quotients will be defined in terms of projections:

\begin{definition}\cite{Handbook-Abraham} If $\P_1$ and $\P_2$ are posets, a \emph{projection} is an onto map $\pi:\P_1 \to \P_2$ such that:

\begin{itemize}
\item $p \le q$ implies that $\pi(p) \le \pi(q)$,
\item $\pi(1_{\P_1})=1_{\P_2}$,
\item if $r \le \pi(p)$, then there is some $q \le p$ such that $\pi(q) = r$.
\end{itemize}

A projection is \emph{trivial} if $\pi(p) = \pi(q)$ implies that $p$ and $q$ are compatible.
\end{definition}

Trivial projections are basically ismorphisms:

\begin{fact} If $\pi:\P_1 \to \P_2$ is a trivial projection, then $\P_1 \simeq \P_2$, i.e$.$ $\P_1$ and $\P_2$ are forcing-equivalent.\end{fact}

For our purposes, we are interested in the preservation of stationary sets. The chain condition gives us preservation fairly straightforwardly. The following fact is implicit in parts of the literature, and a version of it can be found in this paper in the form of \autoref{long-closed-pres}.

\begin{fact}\label{chain-pres} If $\P$ has the $\mu$-chain condition and $S \subset P_\mu(X)$ is stationary, then $\P$ forces that $S$ is stationary in $P_\mu(X)$.\end{fact}


However, we must place demands on our stationary sets in order for them to be preserved by closed forcings.

\begin{definition} Consider a regular uncountable cardinal $\mu$ and a stationary set $S \subset P_\mu(X)$. We say that $S$ is \emph{internally approachable of length} $\tau$ if for all $N \in S$ with $N \prec H(X)$, there is a continuous chain of elementary submodels $\seq{M_i}{i<\tau}$ such that: $N = \bigcup_{i<\tau}M_i$ and for all $i<\tau$, $\seq{M_i}{i<j} \in N$. In this case we write $S \subseteq \mathcal{IA}(\tau)$.\end{definition}

Here we are following Krueger's convention \cite{Krueger2009}, which withholds the requirement that $|M_i|<\mu$ for $i<\tau$.





\begin{fact}\label{closed-pres} If $S \subset P_\mu(X) \cap \mathcal{IA}(\tau)$ is an internally approachable stationary set, $\tau<\mu$, and $\P$ is $\mu$-closed, then $\P$ forces that $S$ is stationary in $P_\mu(X)^V)$.\footnote{See \cite{Foreman-Magidor1995} for discussion of related facts.}\end{fact}

\subsection{Costationarity of the Ground Model}

The notion of ground model costationarity is a key ingredient in arguments pertaining to disjoint stationary sequences. It will specifically give us the disjointness, since we will be picking stationary sets that are not added by initial segments of these forcings.

Gitik obtained the classical result:

\begin{fact}[Gitik \cite{Gitik1985}] If $V \subset W$ are models of $\ZFC$ with the same ordinals, $W \setminus V$ contains a real, and $\kappa$ is a regular cardinal in $W$ such that $(\kappa^+)^W \le \lambda$, then $P^W_\kappa(\lambda) \setminus V$ is stationary as a subset of $P_\kappa(\lambda)$ in the model $W$.\end{fact}


Because we will need \autoref{closed-pres}, we will actually use Krueger's refinement of Gitik's theorem:

\begin{fact}[Krueger \cite{Krueger2009}]\label{kruegergitik} Suppose $V \subset W$ are models of $\ZFC$ with the same ordinals, $W \setminus V$ contains a real, $\mu$ is a regular cardinal in $W$, and $X \in V$ is such that $(\mu^+)^W \subseteq X$, and that in $W$, $\Theta$ is a regular cardinal such that $X \subset H(\Theta)$. Then in $W$ the set $\{N \in P_\mu(H(\Theta)) \cap \mathcal{IA}(\omega):N \cap X \notin V\}$ is stationary. 
\end{fact}

\section{The New Mitchell Forcing}\label{sec-thm1and2}

\subsection{Defining the Forcing}

In this subsection we will illustrate the basic idea of this paper by using our new take on Mitchell forcing to prove a known result:

\begin{fact}[Krueger \cite{Krueger2009}]\label{kruegertheorem} If $\lambda$ is a Mahlo cardinal and $\mu<\lambda$ is regular, there is a forcing extension in which $2^\omega = \mu^+=\lambda$ and there is a disjoint stationary sequence on $\lambda$.\end{fact}

Specifically, we will define a forcing $\M^+(\tau,\mu,\lambda)$ such that the model $W$ in \autoref{kruegertheorem} can be realized as an extension by $\M^+(\omega,\mu,\lambda)$.

For standard technical reasons, we define a poset ismorphic to $\Add(\tau,\lambda)$:

\begin{definition} Given a regular $\tau$ and a set of ordinals $Y$, we let $\Add^*(\tau,Y)$ be the poset consisting of partial functions $p : \{\delta \in Y: \delta \text{ is inaccessible}\} \times \tau \to \{0,1\}$ where $|\dom p|<\tau$. We let $p \le_{\Add^*(\tau,Y)} q$ if and only if $p \supseteq q$.\end{definition}

\emph{Note:} In later subsections we will conflate $\Add(\tau,\lambda)$ and $\Add^*(\tau,\lambda)$ to simplify notation.


\begin{definition} Let $\lambda$ be inaccessible and let $\tau < \mu<\lambda$ be regular cardinals such that $\tau^{<\tau} = \tau$. We define  a forcing $\M^+(\tau,\mu,\lambda)$ that consists of pairs $(p,q)$ such that:

\begin{enumerate}

\item $p \in \Add^*(\tau,\lambda)$,

\item $q$ is a function such that:

\begin{enumerate}

\item $\dom q \subset \{\delta < \lambda: \delta \text{ is inaccessible}\}$,

\item $|\dom q|<\mu$,

\item $\forall \delta \in \dom(q)$, $q(\delta)$ is an $\Add^*(\tau,\delta+1)$-name such that $p \nrest ((\delta + 1) \times \tau) \Vdash_{\Add^*(\tau,\delta+1)}``q(\delta) \in \dot{\Col}(\mu,\delta)$''.

\end{enumerate}
\end{enumerate}

We let $(p,q) \le (p',q')$ if and only if:

\begin{enumerate}[(i)]

\item $p \le_{\Add^*(\tau,\lambda)} p'$,

\item $\dom q \supseteq \dom q'$,

\item for all $\delta \in \dom q', p \nrest ((\delta+1) \times \tau) \Vdash_{\Add^*(\tau,\delta+1)}``q(\delta) \le_{\dot{\Col}(\mu,\delta)}q'(\delta)$''

\end{enumerate}
\end{definition}

First we go through the more routine properties that one would expect of this forcing.

\begin{proposition}  $\M^+(\tau,\mu,\lambda)$ is $\tau$-closed and $\lambda$-Knaster.\end{proposition}

\begin{proof} Closure uses the facts that $\Add^*(\tau,\lambda)$ is $\tau$-closed and $\Vdash_{\Add^*(\tau,\delta+1)}``\dot{\Col}(\mu,\delta)$ is $\mu$-closed'' for all $\delta$. For Knasterness: Consider $\{(p_i,q_i):i<\lambda\} \subseteq \M^+(\tau,\mu,\lambda)$, then fix a regular $\rho \in (\mu,\lambda)$ and find a stationary subset of $\lambda \cap \cof(\rho)$ on which $\dom(p_i),\dom(q_i)$ are fixed, and then proceed with a standard Delta System Lemma argument.\end{proof}

Crucially, we get a nice termspace:

\begin{definition} Let $\T=\T(\M^+(\tau,\mu,\lambda))$ be the poset consisting of conditions $q$ such that:

\begin{enumerate}

\item $\dom q \subset \{\delta < \lambda: \delta \textup{ is inaccessible}\}$,

\item $|\dom q| <\mu$,

\item $\forall \delta \in \dom q, \Vdash_{\Add^*(\tau,\delta+1)}q(\delta) \in \dot{\Col}(\mu,\delta)$''.

\end{enumerate}

Most importantly, we let $q \le q'$ if and only if:

\begin{enumerate}[(i)]

\item $\dom q \supseteq \dom q'$,

\item for all $\delta \in \dom q$, $\Vdash_{\Add^*(\tau,\delta+1)} ``q(\delta) \le q'(\delta)$''.

\end{enumerate}
\end{definition}

\begin{proposition} There is a projection $\Add^*(\tau,\lambda) \times \T(\M^+(\tau,\mu,\lambda)) \surj \M^+(\tau,\mu,\lambda)$.\end{proposition}

\begin{proof} We let $\pi$ be the projection with the definition $\pi(p,q) = (p,q)$. This is automatically order-preserving because the ordering $\le_{\Add^*(\tau,\lambda) \times \T}$ is coarser than the ordering $\le_{\M^+(\tau,\mu,\lambda)}$. For obtaining the density condition, suppose $(r,s) \le_{\M^+(\tau,\mu,\lambda)} (p_0,q_0)$. We want to find some $(p_1,q_1)$ such that $(p_1,q_1) \le_{\Add^*(\tau,\lambda) \times \T} (p_0,q_0)$ and $(p_1,q_1) \le_{\M^+(\tau,\mu,\lambda)} (r,s)$. To do this, we first let $p_1 = r$, and then we define $q_1$ with $\dom q_1 = \dom r$ such that at each coordinate $\delta \in \dom q_1$, we use standard arguments on names to show that we can get both $p_0 \nrest ((\delta+1) \times \tau) \Vdash_{\Add^*(\tau,\lambda)} ``q_1(\delta) \le s(\delta)$'' as well as $1_{\Add^*(\tau,\lambda)} \Vdash_{\Add^*(\tau,\lambda)}``q_1(\delta) \le q_0(\delta)$''.\end{proof}

\begin{proposition} $\T=\T(\M^+(\tau,\mu,\lambda))$ is $\mu$-closed. \end{proposition}

\begin{proof} This is an application of the Mixing Principle. Given a $\le_\T$-decreasing sequence $\seq{q_i}{i<\tau}$ with $\tau<\mu$ we let $d = \bigcup_{i<\tau}\dom q_i$. Then we define a lower bound $\bar q$ with domain $d$ such that for all $\delta \in d$, $q(\delta)$ is a canonically-defined name for a lower bound of the $q_i(\delta)$'s (where $i$ is large enough that $\delta \in \dom q_i$).\end{proof}

Then we get the standard consequences of the termspace analysis:

\begin{proposition} The following are true in any extension by $\M^+(\tau,\mu,\lambda)$:

\begin{enumerate}
\item $V$-cardinals up to and including $\mu$ are cardinals.
\item For all $\alpha<\lambda$, $|\alpha|=\mu$.
\item $\lambda = \mu^+$.
\item $2^\tau = \lambda$.
\end{enumerate}
\end{proposition}

\begin{proof} \emph{(1)} follows from the projection analysis and the fact that $\T$ is $\mu$-closed and $\Add^*(\tau,\lambda)$ is $\tau^+$-cc, and from $\tau$-closure of $\M^+(\tau,\mu,\lambda)$. \emph{(2)} follows from the fact that for all inaccessible $\delta<\lambda$, $\M^+(\tau,\mu,\lambda)$ projects onto the termspace forcing $\T(\M^+(\tau,\mu,\lambda))$. \emph{(3)} follows from \emph{(1)} and \emph{(2)} plus $\lambda$-Knasterness. \emph{(4)} follows from the fact that $\M^+(\mu,\lambda)$ projects onto $\Add^*(\tau,\lambda)$, so it forces that $2^\tau \ge \lambda$. Since the poset has size $\lambda$ and $\lambda$ is inaccessible, it also forces that $2^\tau \le \lambda$.\end{proof}

%
%
%
%
%
%
%
%

The following lemma is the crux of the new idea.

\begin{lemma}\label{mainlemma} If $\delta_0 < \lambda$ is inaccessible, then there is a forcing equivalence
\[
\M^+(\tau,\mu,\lambda) \simeq \M^+(\tau,\mu,\delta_0) \ast \Add(\tau) \ast \mathbb{E}
\]
where $\M^+(\tau,\mu,\delta_0) \ast \Add(\tau)$ forces that $\mathbb{E}$ is a projection of a product of a $\mu$-closed forcing and a $\tau^+$-cc forcing.\end{lemma}


%
%

\begin{proof} More precisely, we will show that there is a forcing equivalence $\M^+(\tau,\mu,\lambda) \simeq \M^+(\tau,\mu,\delta_0) \ast \Add(\tau) \ast (\mathbb{F} \times \mathbb{G})$ where the following hold in the extension by $\M^+(\tau,\mu,\delta_0) \ast \Add(\tau)$:

\begin{itemize}
\item $\mathbb{G}$ is a projection of a product of a $\mu$-closed forcing and $\Add^*(\tau,\lambda)$, and
\item $\mathbb{F}$ is $\mu$-closed.
\end{itemize}

The statement of the lemma can then be obtained by merging $\mathbb{F}$ with the closed component of the product that projects onto $\mathbb{G}$.

First we describe $\mathbb{F}$ and $\mathbb{G}$. To do this, we fix some notation. Given $Y \subseteq \lambda$, we let $\pi_\Add^Y$ denote the projection $(p,q) \mapsto p \nrest (Y \times \tau)$ from $\M^+(\tau,\mu,\lambda)$ onto $\Add^*(\tau,Y)$. For any poset $\P$, we employ the convention that $\Gamma(\P)$ denotes a canonical name for a $\P$-generic. If $X \subset \P$, then we use the notation $\uparrow X :=\{q \in \P : \exists p \in X, p \le q\}$.

We will let
\[
\mathbb{F}:= \Col(\mu,\delta_0)^{V[(\uparrow( \pi_\Add^{\delta_0}"\Gamma(\M^+(\tau,\mu,\delta_0)))) \times \Gamma(\Add(\tau))]}
\]
if we are working in an extension by $\M^+(\tau,\mu,\delta_0) \ast \Add(\tau)$. (In other words, the poset $\mathbb{F}$ will be the version of $\Col(\mu,\delta_0)$ as interpreted in the extension of $V$ by $\Add^*(\tau,\delta_0+1)$ where the initial coordinates come from $\M^+(\tau,\mu,\delta_0)$ and the last coordinate comes from the additional copy of $\Add(\tau)$ that occupies the coordinate $\delta_0$ in $\Add^*(\tau,\delta_0+1)$.)

Still working in an extension by $\M^+(\tau,\mu,\delta_0) \ast \Add(\tau)$, the poset $\mathbb{G}$ consists of pairs $(p,q)$ such that the following hold:

\begin{enumerate}

\item $p \in \Add^*(\tau,(\delta_0,\lambda))$,

\item $q$ is a function such that

\begin{enumerate}[(a)]


\item $\dom q \subset \{\delta \in (\delta_0,\lambda):\delta \text{ is inaccessible}\}$,

\item $|\dom q| < \mu$,

\item $\forall \delta \in \dom(q),p \nrest ((\delta_0,(\delta + 1)) \times \tau) \Vdash_{\Add^*(\tau,(\delta_0,\delta+1))}``q(\delta) \in \dot{\Col}(\mu,\delta)$''.

\end{enumerate}
\end{enumerate}

The ordering is the one analogous to that of $\M^+(\tau,\mu,\lambda)$. An easy adaptation of the arguments for the projection analysis for $\M^+(\tau,\mu,\lambda)$ will then give a projection analysis for $\mathbb{G}$.

The rest of the proof of the lemma consists of verifying the more substantial claims.

\begin{claim} $ \M^+(\tau,\mu,\lambda) \simeq \M^+(\tau,\mu,\delta_0) \ast \Add(\tau,1) \ast ( \mathbb{F} \times \mathbb{G})$.\end{claim}

\begin{proof} We identify $\M^+(\tau,\mu,\delta_0) \ast \Add(\tau,1) \ast ( \mathbb{F} \times \mathbb{G})$ with the dense subset of conditions $((r,s),t,u,(\dot{r}',\dot{s}'))$ such that $\dot{s}'$ is forced to have a specific domain in $V$. The fact that this subset is dense follows from the fact that $\M^+(\tau,\mu,\lambda) \ast \Add(\tau,1)$ has the $\mu$-covering property.

We will argue that there is a trivial projection defined by
\[
\pi:(p,q) \mapsto (\underbrace{(p \nrest (\delta_0 \times \tau),q \nrest \delta_0)}_{\M^+(\mu,\delta_0)},\underbrace{p \nrest (\{\delta_0\} \times \tau)}_{\Add(\tau)}, \underbrace{q^*(\delta_0)}_{\mathbb{F}},\underbrace{(\bar{p},\bar{q})}_{\mathbb{G}})
\]
such that
\begin{itemize}
\item $\bar p:= p \nrest ((\delta_0,\lambda) \times \tau)$;
\item $q^*(\delta_0)$ is obtained by changing $q(\delta_0)$ from an $\Add^*(\tau,\delta_0+1)$-name to an $\Add(\tau)$-name as interpreted in the extension by the relevant generic, namely $(\uparrow( \pi_\Add^{\delta_0}"\Gamma(\M^+(\tau,\mu,\delta_0))))$;
\item $\bar q$ has domain $(\delta_0,\lambda)$, and for each $\delta \in (\delta_0,\lambda)$, $\bar q(\delta)$ has changes analogous to the changes made to $q^*(\delta_0)$. 
\end{itemize}

It is clear that $\pi$ is order-preserving. We also want to show that if
\[
((r,s),t,u,(\dot{r}',\dot{s}')) \le_{\M^+(\tau,\mu,\delta_0) \ast \Add(\tau) \ast (\mathbb{F} \times \mathbb{G})} \pi(p_0,q_0)
\]
then there is some $(p_1,q_1) \le_{\M^+(\mu,\lambda)} (p_0,q_0)$ such that we have $\pi(p_1,q_1) \le ((r,s),t,u,(\dot{r}',\dot{s}'))$. This can be done by taking:

\begin{itemize}
\item $p_1 =r^* \cup \tilde{t} \cup r'$ where $r^* \le r$ decides $t$ and $\dot{r}'$ and $\tilde{t}$ writes $t$ as as a partial function $\{\delta\} \times \tau \to \{0,1\}$,
\item $q_1 = s \cup \tilde{u} \cup \tilde{s}'$ where $\tilde u$ reinterprets $u$ as a $\Add^*(\delta_0+1)$-name and for each $\delta \in \dom(\dot{s}')$, $\tilde{s}'$ reinterprets $\dot{s}'(\delta)$ as a $\Add^*(\delta+1)$-name.
\end{itemize}

Last, we argue that $\pi(p_0,q_0)=\pi(p_1,q_1)$ implies that $(p_0,q_0)$ and $(p_1,q_1)$ are compatible. Suppose that $(p_0,q_0)$ and $(p_1,q_1)$ are incompatible. If $p_0$ and $p_1$ are incompatible as elements of $\Add^*(\tau,\lambda)$, then one of $p_i \nrest (\delta_0 \times \tau)$, $p_i \nrest (\{\delta_0\} \times \tau)$, and $p_i \nrest ((\delta_0,\lambda) \times \tau)$ must be distinct for $i=0$ and $i=1$. Otherwise, there is some $p' \le p_0,p_1$ and some $\delta \in \dom q_0 \cap \dom q_1$ inaccessible such that $p' \Vdash ``q_0(\delta) \perp q_1(\delta)$'', which implies that $q_0(\delta) \ne q_1(\delta)$. Therefore, one of $q_i \nrest \delta_0$, $q_i(\delta_0)$, or $q_i \nrest (\delta_0,\lambda)$ is distinct for $i \in \{0,1\}$. \end{proof}

\begin{claim} $V[\M^+(\tau,\mu,\delta_0)][\Add(\tau,1)] \models ``\mathbb{F}$ is $\mu$-closed''.\end{claim}


\begin{proof} In fact, our argument will also show that
\[
V[\M^+(\tau,\mu,\delta_0)][\Add(\tau,1)] \models `` \mathbb{F} = \Col(\mu,\delta_0)\text{''}.
\]
We fix some arbitrary generics:

\begin{itemize}
\item $G$ is $\M^+(\tau,\mu,\delta_0)$-generic over $V$,
\item $r$ is $\Add(\tau,1)$-generic over $V[G]$,
\item $H$ is the $\Add^*(\tau,\delta_0)$-generic induced from $G$ by $\pi_\Add^{\delta_0}$,
\item $K$ is the generic for the quotient of $\M^+(\tau,\mu,\delta_0)$ by $\Add^*(\tau,\delta_0)$, i.e$.$ the generic such that $V[H][K]=V[G]$,
\item $T$ is the generic for the termspace forcing $\T(\M^+(\tau,\mu,\delta_0))$, so that $V[G] \subset V[T][H]$.
\end{itemize}

It is enough to argue that $V[G][r] \models ``\mathbb{F}$ is $\mu$-closed'' knowing that $V[H][r] \models ``\mathbb{F}$ is $\mu$-closed''. Because adjoining $G$ does not change the definition of $\Add(\tau,1)$, and because $K$ is defined in terms of the subsets of $\tau$ adjoined by the filter $H$, we have $V[G][r] = V[H][K][r]=V[H][r][K]$. Therefore, it is enough to show that $K$ does not add $<\!\mu$-sequences over $V[H][r]$, so that $V[H][r]$'s version of $\Col(\mu,\delta_0)$ remains $\mu$-closed in $V[G][r]$. We have
\begin{multline*}
V[H][r] \subset V[H][r][K] = V[H][K][r]  = \\ = V[G][r] \subset V[T][H][r] = V[H][r][T].
\end{multline*}
Recall Easton's Lemma, which states in part that if $\mathbb{A}$ is $\mu$-cc and $\mathbb{B}$ is $\mu$-closed, then $\Vdash_{\mathbb{A}}\textup{``}\mathbb{B}$ is $\mu$-distributive$\textup{''}$. Easton's Lemma implies that $T$ does not add new $<\!\mu$-sequences over $V[H][r]$ since the forcing adjoining $r$ is $\mu$-cc over $V[H]$ and the forcing adjoining $T$ is $\mu$-closed over $V[H]$. Therefore $K$ does not add new $<\!\mu$-sequences over $V[H][r]$ since it is an intermediate factor of the extension.\end{proof}

This completes the proof of the lemma.\end{proof}

Now we have an application for the case where $\tau = \omega$.

\begin{proposition}\label{wehaveit} If $\lambda$ is Mahlo then $\Vdash_{\M^+(\omega,\mu,\lambda)} \DSS(\lambda)$.\end{proposition}

This basically repeats Krueger's argument for \cite[Theorem 9.1]{Krueger2009}.

\begin{proof} Let $G$ be $\M^+(\omega,\mu,\lambda)$-generic over $V$. The set of $V$-inaccessibles in $\lambda$ will form the stationary set $S \subset \mu^+ \cap \cof(\mu)$ carrying the disjoint stationary sequence in the extension by $\M^+(\omega,\mu,\lambda)$. For every such $\delta \in S$, let $\bar{G}$ be the generic on $\M^+(\omega,\mu,\delta)$ induced by $G$ and let $r$ be the $\Add(\omega)$-generic induced by $G$ via $\pi_\Add^{\{\delta\}}$. We use \autoref{kruegergitik} to obtain a stationary set $\mathcal{S}^*_\delta \subset P_\mu(H(\delta))^{V[\bar G][r]}$ such that for all $N \in \mathcal{S}^*_\delta$, $N \cap \delta \notin V[\bar G]$ and such that $N$ is also internally approachable by a $\omega$-sequence. Therefore we can apply \autoref{mainlemma} with \autoref{closed-pres} and then \autoref{chain-pres} to find that $\mathcal{S}_\delta^*$ is stationary in $V[G]$. We then let $\mathcal{S}_\delta = \{N \cap \delta:N \in \mathcal{S}_\delta^*\}$, and we see that $\seq{\mathcal{S}_\delta}{\delta \in S}$ is a disjoint stationary sequence.\end{proof}

\subsection{Proving the Main Theorems}\label{rezults}

Now we will apply the new version of Mitchell forcing to answer Krueger's questions. We can readily prove \autoref{consecutivetheorem}, which states the can obtain $\DSS(\aleph_2) \wedge \DSS(\aleph_3)$:

\begin{proof}[Proof of \autoref{consecutivetheorem}] Begin with a ground model $V$ in which $\lambda_1<\lambda_2$ and the $\lambda$'s are Mahlo. Let $\M_1=\M^+(\omega,\aleph_1,\lambda_1)$. (Any $\lambda_1$-sized forcing that turns $\lambda_1$ into $\aleph_2$ and adds a disjoint stationary sequence on $\aleph_2$ would work, so we could also use a more standard mixed support iteration.) Then let $\dot{\M}_2$ be an $\M_1$-name for $\M^+(\omega,\lambda_1,\lambda_2)$. We argue that if $G_1$ is $\M_1$-generic over $V$ and $G_2$ is $\dot{\M}_2[G_1]$-generic over $V[G_1]$, then $V[G_1][G_2] \models ``\DSS(\lambda_1) \wedge \DSS(\lambda_2)$''. We get $\DSS(\lambda_2)$ from the fact that $\lambda_2$ remains Mahlo in $V[G_1]$ together with \autoref{wehaveit}, so we only need to argue that the disjoint stationary sequence $\vec{\mathcal{S}}:=\seq{\mathcal{S}_\alpha}{\alpha \in S} \in V[G_1]$ remains a disjoint stationary sequence in $V[G_1][G_2]$.

Working in $V[G_1]$, preservation of $\vec{\mathcal{S}}$ follows from the projection analysis: Let $H_1$ and $H_2$ be chosen so that $H_1$ is $\T:=\T(\M_2)$-generic over $V[G_1]$, $H_2$ is $\Add(\omega,\lambda_2)^{V[G_1]}$-generic over $V[G_1][H_1]$, and $V[G_1][G_2] \subseteq V[G_1][H_1][H_2]$. Since $\T$ is $\lambda_1$-closed, it preserves stationarity of $S$ and the $\mathcal{S}_\alpha$'s, and $\Add(\omega,\lambda_2)^{V[G_1]}$ still has the countable chain condition in $V[G_1][H_1]$. It follows that the stationarity of $S$ is preserved in $V[G_1][H_1][H_2]$, as well as the stationarity of the $\mathcal{S}_\alpha$'s (by \autoref{chain-pres}). Therefore $\vec{\mathcal{S}}$ is a disjoint stationary sequence on $\lambda_1$ in $V[G_1][G_2]$.\end{proof}


It will take a bit more work to show obtain \autoref{superconsecutivetheorem} holds in the same model for \autoref{consecutivetheorem}. (Recall that \autoref{superconsecutivetheorem} states that we can simultaneously separate internally stationary and internally club for $[H(\aleph_2)]^{\aleph_1}$ and $[H(\aleph_3)]^{\aleph_2}$.) Note that we cannot just apply \autoref{dss-equivalence} because $2^\omega = \aleph_3$ in the model for \autoref{consecutivetheorem}, plus it is consistent that there can be a stationary set which is internally unbounded but not internally stationary \cite{Krueger2008}.


We will give some facts on preservation of the distinction between stationary sets that are internally stationary but not internally club:

\begin{proposition}\label{long-closed-pres} Suppose $\P$ is $\nu$-closed and $S \subseteq [X]^\delta$ is a stationary set such that $|[X]^\delta| \le \nu$ and $|X|>1$. Then $\Vdash_\P ``S$ is stationary in $[X]^\delta\text{''}$.\end{proposition}


\begin{proof} Let $\dot{C}$ be a $\P$-name for a club in $[X]^\delta$ and let $\vec{x}=\seq{x_\xi}{\xi \le \bar{\nu}}$ enumerate $[X]^\delta$ (where $\bar{\nu} \le \nu$). Note that we have $\delta < 2^\delta \le |X|^\delta \le \nu$, so conditions in $\P$ can decide names for elements of $\dot{C}$. We construct a sequence $\vec{z}=\seq{z_\xi}{\xi<\bar{\nu}} \subseteq [X]^\delta$ and a $\le_\P$-descending sequence $\seq{p_\xi}{\xi<\bar{\nu}}$ using the closure of $\P$ such that for all $\xi<\bar{\nu}$, $p_\xi \Vdash \textup{``}x_\xi \subseteq z_\xi \in \dot{C}\textup{''}$ and $p_\xi \| \textup{``}x_\xi \in \dot{C}\textup{''}$.

Then let $D$ be the set $\{x_\xi : \exists \zeta<\bar{\nu},p_\zeta \Vdash \textup{``}x_\xi \in \dot{C}\textup{''}\}$. We can argue that $D$ is a club: It is unbounded because of the sets chosen for $z_\xi$. It is closed because if $\seq{x_{\xi_i}}{i<\bar{\delta}} \subseteq D$ (for $\bar{\delta}\le \delta$) is an $\subseteq$-increasing sequence such that we have $p_{\zeta_i} \Vdash \textup{``}x_{\xi_i} \in \dot{C}\textup{''}$, and $\zeta^* = \sup_{i<\bar{\delta}}\zeta_i$, then $p_{\zeta^*} \Vdash \textup{``}\bigcup_{i<\bar{\delta}}x_{\xi_i} \in \dot{C}\textup{''}$.

There is some $w \in D \cap S$. If $p_\xi$ is such that $p_\xi \Vdash \textup{``}w \in \dot{C}\textup{''}$, then we have $p_\xi \Vdash \textup{``}\dot{C} \cap S \ne \emptyset\textup{''}$.\end{proof}


\begin{proposition}\label{distinction-pres} Suppose $|[H(\theta)]^\delta| \le \nu$. Let $\mathbb{F}$ have the $\delta$-chain condition and let $\mathbb{G}$ be $\nu$-closed. If there is a stationary set $S \subseteq [H(\theta)]^\delta$ consisting of sets that are internally stationary but not internally club. Then $\mathbb{F} \times \mathbb{G}$ forces that there is a stationary set consisting of sets that are internally stationary but not internally club.\end{proposition}

\begin{proof} Since $\mathbb{G}$ preserves the chain condition of $\mathbb{F}$, we show that preservation of the distinction can be achieved by forcing with $\mathbb{G}$ and then $\mathbb{F}$. The poset $\mathbb{G}$ preserves the distinction by \autoref{long-closed-pres} and the fact that it does not change $H(\theta)$.

Now we argue that $\mathbb{F}$ preserves the distinction. Let $S$ be the witnessing stationary set in $V$ and let $X = H(\theta)^V$. If $G$ is $\mathbb{F}$-generic over $V$, let $Y = H(\theta)^{V[G]}$ and let $S^*=\{M \in [Y]^\delta:M \cap X \in S\}$. We will argue that $S^*$ witnesses the relevant statement in $V[G]$. Let $\dot{S}^*$ be a name for $S^*$.

To see that $\dot{S}^*$ is forced to be stationary, let $\dot{C}$ be a name for a club in $[\dot{Y}]^\delta$. Given $p \in \mathbb{F}$, let $D = \{z:\exists \dot{w},p \Vdash \textup{``}\dot{w} \in \dot{C},\dot{w} \cap X = z\}$ is a club in $[X]^\delta$ as regarded in $V$, so there is some $z \in S$, and hence $p \Vdash \textup{``}\dot{w} \in \dot{C} \cap \dot{S}^*\textup{''}$.

Next we argue that members of $\dot{S}^*$ are forced not to be internally club. Suppose for contradiction, then, that $p$ forces $\dot{M} \in \dot{S}^*$ to be internally club as witnessed by $\dot{c}$, and also that $N = \dot{M} \cap X$ where $N \in S$. Let $d=\{z:\exists \dot{w},p \Vdash \textup{``}\dot{w} \in \dot{c},\dot{w} \cap X=z\}$. Then $d$ is a club in $P_\mu(N)$ since if $a \in z \subseteq N$ then $a \in N$ and if $p \Vdash \textup{``}\dot{w} \cap X = z\textup{''}$ then in particular $p \Vdash \textup{``}\dot{w} \in \dot{M}\textup{''}$, so $z \in N$. This contradicts the fact that $N$ consists of sets that are not internally club.

Finally, we argue that $\dot{S}^*$ is forced to be internally stationary. Let $\dot{M}$ be forced by $p$ to be in $\dot{S}^*$ and that $N=\dot{M} \cap X$. Let $G$ be generic with $p \in G$ and work in $V[G]$. Then $\{w \subseteq P_\delta(M):\exists z \in N,w = z[G]\}$ is a club as regarded in $V[G]$. As in the argument for stationarity, any name $\dot{c}$ for a club in $P_\delta(\dot{M})$ can produce a corresponding club $d$ in the ground model. Then we can find some $z \in N \cap d$ and if $G$ is generic with $p \in G$ then $z[G] \in c \cap M$.\end{proof}



We use a concept from Harrington and Shelah to handle Mahlo cardinals \cite{Harrington-Shelah1985}:

\begin{definition} Let $\lambda$ be Mahlo and let $\mathcal{N}$ be a model of some fragment of $\ZFC$. We say that $\mathcal{M} \prec \mathcal{N}$ is \emph{rich} if the following hold: 

\begin{enumerate}
\item $\lambda \in \mathcal{M}$;
\item $\bar{\lambda}:=\mathcal{M} \cap \lambda \in \lambda$;
\item $\bar \lambda$ is an inaccessible cardinal in $\mathcal{N}$;
\item The size of $\mathcal{M}$ is $\bar{\lambda}$;
\item $\mathcal{M}$ is closed under $<\!\bar{\lambda}$-sequences and $\bar{\lambda}<\lambda$.
\end{enumerate}
\end{definition}

\begin{lemma}\label{distinctionlemma} If $\lambda$ is Mahlo, then $\M^+(\omega,\mu,\lambda)$ forces that there are stationarily many $Z \in [\mu^+]^\mu$ which are internally stationary but not internally club.\end{lemma}

This follows Krueger's proof of \cite[Theorem 10.1]{Krueger2009}, making necessary changes for Mahlo cardinals, and including enough details to show that we can get the necessary preservation of stationarity simply from the projection analysis. We do not need guessing functions (which are used in Krueger's argument) because we are only obtaining one instance of separation per large cardinal.

\begin{proof}[Proof of \autoref{distinctionlemma}] Denote $\M:=\M^+(\omega,\mu,\lambda)$ and let $\dot{C}$ be an $\M$-name for a club in $([H(\mu^+)]^\mu)^{V[\M]}$. We want to find an $\M$-name $\dot{Z}$ for an element of $([H(\mu^+)]^\mu)^{V[\M]} \cap \dot{C}$ that is internally stationary but not internally club. Let $\dot{F}$ be an $\M$-name for a function $(H(\mu^+)^{V[\M]})^{<\omega} \to H(\mu^+)^{V[\M]}$ with the property that all of its closure points are in $\dot{C}$. Let $\Theta$ be as large as needed for the following discussion and let $\mathcal{N}$ be the structure $(H(\Theta),\in,<_\Theta,\M,\dot{F},\lambda,\mu)$ where $<_\Theta$ is a well-ordering of $H(\Theta)$.

Since $\lambda$ is Mahlo, we can find some $\mathcal{K} \prec \mathcal{N}$ with $\mu \subset \mathcal{K}$ that is a rich submodel of cardinality $\bar{\lambda}$. Now set $G$ to be $\M$-generic over $V$. Note that $H(\lambda)^{V[G]} = H(\lambda)[G]$ because $\M$ has the $\lambda$-chain condition and $\M \subset H(\lambda)$. We will argue that $Z:=\mathcal{K}[G] \cap H(\lambda)[G]$ is what we are looking for.

\begin{claim} $Z \in C:=\dot{C}[G]$.\end{claim}

\begin{proof} We have $\bar \lambda \le |Z| \le |\mathcal{K}| \le \bar \lambda$ and $\bar \lambda$ has cardinality $\mu$ in $\mathcal{N}[G]$, so $Z \in [H(\lambda)^{V[G]}]^\mu$. If $a_1,\ldots,a_n \in Z$, there are $\M$-names $\dot{b}_1,\ldots,\dot{b}_n \in \mathcal{K} \cap H(\lambda)$ such that $a_i=\dot{b}_i[G]$ for all $1 \le i \le n$. By elementarity, $\mathcal{K}$ contains the $<_\Theta$-least maximal antichain $A \subset \M$ of conditions deciding $\dot{F}(\dot{b}_1,\ldots,\dot{b}_n)$. Since $|A|<\lambda$, $|A| \in \mathcal{K} \cap \lambda = \bar \lambda$, so it will follow that $A \subset \mathcal{K}$. Therefore if $p \in G \cap A$, then $p \in M$ in particular, so $p \Vdash \dot{F}(\dot{b}_1,\ldots,\dot{b}_n) = \dot{b}_\ast$ for some $\dot{b}_\ast \in \mathcal{K} \cap H(\lambda)$ where we automatically get $\dot{b}_\ast \in H(\bar \lambda)$, and therefore $F(a_1,\ldots,a_n) = a_\ast := \dot{b}_\ast[G] \in \mathcal{K}[G] \cap H(\lambda)[G] = Z$ (where of course $F:=\dot{F}[G]$).\end{proof}

For the rest of the proof let $\bar G:= \pi_\mathcal{K}(G)$ where $\pi_{\mathcal{K}}$ is the Mostowski collapse relative to $\mathcal{K}$. Since $\pi_\mathcal{K}(\M)=\M^+(\omega,\mu,\bar \lambda)$, there is an extension $\pi_{\mathcal{K}}:\mathcal{K}[G] \cong \pi_{\mathcal{K}}(\mathcal{K})[\bar G]$. We also denote $h:=\pi_{\mathcal{K}}(H(\lambda)[G] \cap \mathcal{K}[G])$.


\begin{claim} $Z$ is internally stationary.\end{claim}

\begin{proof} First, we argue that $S:=P_\mu(h)^{\mathcal{N}[\bar G]}$ is stationary as a subset of $P_\mu(h)^{\mathcal{N}[G]}$ in $\mathcal{N}[G]$. By \autoref{mainlemma}, the quotient $\M/\bar{G}$ is a projection of a forcing of the form $\A_1 \ast (\dot{\T} \times \A_2)$ where $\A_1$ has the countable chain condition, $\dot{\T}$ is an $\A_1$-name for a $\mu$-closed forcing, and $\A_2$ also has the countable chain condition. Let $K_1$, $K_T$, and $K_2$ be respective generics such that $V[G] \subseteq V[\bar G][K_1][K_T][K_2]$. Working in $\mathcal{N}[\bar G]$, note that $S':= S \cap \mathcal{IA}(\omega)$ is stationary, and therefore has its stationarity preserved in $V[\bar G][K_1]$ by \autoref{chain-pres}.

We must also show that the stationarity of $S'$ will be preserved by countably closed forcings over $\mathcal{N}[\bar G][K_1]$. Suppose $\seq{M_n}{n<\omega}$ witnesses internal approachability of some $N \in S'$ in $V[\bar G]$ with respect to the structure $H(\lambda^+)^{V[\bar G]}$, and let $M_\omega :=\bigcup_{n<\omega}M_n$. Then we can see that $\seq{M_n[K_1]}{n<\omega}$ is a chain of elementary submodels of $H(\lambda)[\bar G][K_1] = H(\lambda)^{V[\bar G][K_1]}$. We also have $M_n[K_1] \cap V[\bar G]=M$ and $M_\omega[K_1] \cap V[\bar G] = M_\omega \in S'$ with $M_\omega[K_1] \prec H(\lambda)^{V[\bar G][K_1]}$. If we choose the $M_n$'s to be elementary substructures of $H(\lambda^+)^{V[\bar G]}(\in,<^*,\dot C,\ldots)$ where $<^*$ is a well-ordering and $\dot C$ is a $\A_1 \ast \dot{\T}$-name for a club, then an argument almost exactly like the one showing that internal approachability is preserved (i.e$.$ the proof of \autoref{closed-pres}) will show that $S'$ is stationary in $\mathcal{N}[\bar G][K_1][K_T]$.

Then the extension of $\mathcal{N}[\bar G][K_1][K_T][K_2]$ over $\mathcal{N}[\bar G][K_1][K_T]$ preserves the stationarity of $S'$ by another application of \autoref{chain-pres}, so we get stationarity in $\mathcal{N}[G]$.


Now that we have established preservation of stationarity of $S'$, we can finish the argument. Since $|h|=\mu$ in $\mathcal{N}[G]$, we can write $h=\bigcup_{i<\mu}x_i$ where $\seq{x_i}{i<\mu}$ is a continuous and $\subset$-increasing chain of elements of $P_\mu(h)$. (This is \emph{not} a chain through $P_\mu(h)^{\mathcal{N}[\bar{G}]}$.) The chain is a club in $P_\mu(h)^{\mathcal{N}[G]}$, in which $S'$ is stationary, so there is a stationary $X \subseteq \mu$ such that $\{x_i:i \in X\} \subseteq S'$. Since $S' \subseteq S$, it follows that $i \in X$ implies that $x_i = \pi_{\mathcal{K}}(y_i)$ for some $y_i \in Z$. Therefore $\seq{y_i}{i<\mu}$ is $\subset$-increasing with union $Z$, and in particular $\seq{y_i}{i \in X}$ is stationary in $Z$.\end{proof}

\begin{claim} $Z$ is \emph{not} internally club.\end{claim}

\begin{proof} Suppose for contradiction that $Z$ is internally club and hence that there is a $\subset$-increasing and continuous chain $\seq{Z_i}{i<\mu} \in \mathcal{N}[G]$ with $|Z_i|<\mu$ for all $i<\mu$ and $\bigcup_{i<\mu}Z_i = Z$. So for all $i<\mu$, $Z_i \subset Z$, and so $\seq{\pi_{\mathcal{K}}[Z_i]}{i<\mu}$ is an $\subset$-increasing and continuous chain with union $h$. If we let $W_i:=\pi_\mathcal{K}[Z_i]$ for all $i<\mu$, then the fact that $|W_i|<\mu$ implies that $W_i=\pi_\mathcal{K}(Z_i) \in \mathcal{K}[\bar{G}]$. Therefore $\seq{W_i}{i<\mu}$ is a continuous and $\subset$-increasing chain of sets in $P_\mu(h)$ with union $h$.

Next we define a set $U \in \mathcal{N}[\bar G][r]$ (where $r$ is the generic induced by $G$ from $\pi_\Add^{\{\bar \lambda\}}$) as
\[
\{A \in P_\mu(H(\chi)) \cap \mathcal{IA}(\omega):A \cap h \notin \mathcal{N}[\bar G]\}.
\]
We have a real in $\mathcal{N}[\bar G][r] \setminus \mathcal{N}[\bar G]$ and $(\mu^+)^{\mathcal{N}[\bar G][r]} = \lambda \subset H(\lambda)$. Hence we apply \autoref{kruegergitik} to see that $U$ is stationary in $\mathcal{N}[\bar G][r]$, and it remains stationary in $\mathcal{N}[G]$ by the preservation properties of the quotient (i.e$.$ \autoref{mainlemma} combined with \autoref{closed-pres} and \autoref{chain-pres}). Therefore in $\mathcal{N}[G]$, since $h \subseteq H(\chi)^{\mathcal{N}[\bar{G}][r]}$, $\{A \cap h:A \in U\}$ is stationary in $P_\mu(h)$. Since $\seq{W_i}{i<\mu}$ is club in $h$, there is some $i<\mu$ such that $W_i = A \cap h$ for some $A \in U$. But by definition, $A \cap h \notin \mathcal{N}[\bar G]$, but $W_i \in \mathcal{K}[\bar{G}] \subset \mathcal{N}[\bar{G}]$, so this is a contradiction.\end{proof}

This completes the proof of the lemma.\end{proof}

\begin{proof}[Proof of \autoref{superconsecutivetheorem}] Let $\M_1$ be any $\lambda_1$-sized forcing that turns $\lambda_1$ into $\aleph_2$ and adds stationarily many $N \in [H(\aleph_2)]^{\aleph_1}$ that are internally stationary but not internally club. Let $\dot{\M}_2$ be an $\M_1$-name for $\M^+(\omega,\lambda_1,\lambda_2)$, let $G_1$ be $\M_1$-generic over $V$, and let $G_2$ be $\dot{\M}_2[G_1]$-generic over $V[G_1]$. Then we can see that the theorem holds in $V[G_1][G_2]$: the distinction between internally stationary and internally club on $ [H(\aleph_2)]^{\aleph_1}$ is preserved in $V[G_1][G_2]$ by \autoref{distinction-pres}, and we get a distinction between internally stationary and internally club for $[H(\aleph_3)]^{\aleph_2}$ by \autoref{distinctionlemma}.\end{proof}

\section{A Club Forcing and a Guessing Sequence}\label{sec-thm3}

\subsection{A review of the tools}


The main idea of the proof of \autoref{secondarytheorem} is to force a club through the complement of a canonical stationary set---that is, it is canonical in the sense that it is independent of a particular enumeration used to define it. This set is described as follows:

\begin{fact}[Krueger,\cite{Krueger2009}]\label{canonical-ds} Suppose $\mu$ is an uncountable regular cardinal and $\mu^{<\mu} \le \mu^+$. Let $\underline x = \seq{x_\alpha}{\alpha<\mu^+}$ enumerate $[\mu^+]^{<\mu}$ and let
\[
S(\underline x):= \{\alpha \in \mu^+ \cap \cof(\mu): P_\mu(\alpha) \setminus \seq{x_\beta}{\beta<\alpha} \text{ is stationary}\}.
\]
Then $\DSS(\mu^+)$ holds if and only if $S(\underline x)$ is stationary.\end{fact}

The natural thing to do is to define the following:


\begin{definition} Let $\mu$ be an uncountable regular cardinal such that $\mu^{<\mu} =\mu^+$ and let $\underline x$ and $S(\underline x)$ be defined as in \autoref{canonical-ds}. Then let $\mathcal{C}(\underline x)$ be the set of closed bounded subsets $p$ of $\mu^+$ such that $p \cap S(\underline x) = \emptyset$. We let $p' \le p$ if and only if $p' \cap (\max p + 1) = p$.\end{definition}

\begin{proposition} Assuming $\mu^{<\mu} \le \mu^+$, $\mathcal{C}(\underline x)$ is $\mu^+$-distributive.\end{proposition}

\begin{proof}[Sketch of Proof] If $S(\underline x)$ is nonstationary, then the result is trivial. If it is stationary, then $S(\underline x)$ does not contain a stationary set of approachable points \cite[Corollary 3.7]{Krueger2007}. Since $\mu^{<\mu}\le \mu^+$ there is going to be stationary set $S^*$ of approachable points, which without loss of generality is disjoint from $S(\underline x)$. Then a standard  distributivity argument applies (see Cox's explanation \cite{Cox2021}).\end{proof}

We will also crucially need a characterization of diamonds. This following appears in joint work with Gilton and Stejskalov{\'a} \cite[Lemma 3.12]{Levine-Gilton-Stejskalova2023}.

\begin{fact}\label{mahlolaverfunction} The following are equivalent:

\begin{enumerate}

\item $\lambda$ is Mahlo and $\diamondsuit_\lambda(\textup{Reg})$ (where of course $\textup{Reg}=\{\tau<\lambda:\tau \text{ regular}\}$) holds.

\item There is a function $\ell : \lambda \to V_\lambda$ such that for every transitive structure $\mathcal{N}$ satisfying a rich fragment of $\ZFC$ that is closed under $\lambda^+$-sequences in $V$,  the following holds: For every $A \in \mathcal{N}$ with $A \in H(\lambda^+)$ and any $a \subset \mathcal{N}$ with $|a|<\lambda$, there is a rich $\mathcal{M} \prec \mathcal{N}$ with $a \cup \{\ell\} \cup\{A\} \subset \mathcal{M}$ such that $\ell(\bar \lambda) = \pi_{\mathcal{M}}(A)$ (where $\bar \lambda = \mathcal{M} \cap \lambda$ and $\pi_{\mathcal{M}}$ is the Mostowski collapse).\footnote{The original is stated with a different quantification---for all such $\mathcal{N}$, there exists a function, not the other way around. However, the proof works with the quantification used here.}

\end{enumerate}\end{fact}

We can always use such an $\ell$ assuming the consistency of a Mahlo cardinal: If $\lambda$ is Mahlo in a model $V$, then it is Mahlo in G{\"o}del's class $L$ where $\Diamond_\lambda(S)$ holds for all regular $\lambda$ and stationary $S \subset \lambda$.

Two other forcings will be used, mostly for their black-boxed properties:

\begin{definition} If $T$ is an Aronszajn tree of cardinality $\aleph_1$, let $\B(T)$ be Baumgartner's forcing for specializing Aronszajn trees. It consists of finite functions $f:T \to \omega$ such that $f(x)\ne f(y)$ if $x \le_T y$ or $y \le_T x$. If $f,g \in \B(T)$, then $f \le g$ if and only if $f \supseteq g$.\end{definition}

\begin{definition} Let $S \subset [\aleph_2]^\omega$ be stationary. Then let $\P(S)$ be the forcing consisting of continuous, increasing, and countable chains $\seq{M_\xi}{\xi \le \eta}$ of elements of $S$. For $p,q \in \P(S)$, $p \le q$ if and only if $p$ end-extends $q$.\cite{Friedman-Krueger2007}\end{definition}

\begin{fact}\label{some-facts} The following are true for these forcings:
\begin{enumerate}
\item For Aronszajn trees $T$ of cardinality $\aleph_1$, $\B(T)$ has the countable chain condition.
\item For $S \subset [\aleph_2]^\omega$ stationary, $\P(S)$ adds a closed unbounded set in $[\aleph_2^V]^\omega$ through $S$.
\item If $S \in V$, then $\Add(\omega) \ast \dot{\P}(S)$ has the weak $\omega_1$-approximation property, i.e$.$ if $\dot{f}$ is an $\Add(\omega) \ast \dot{\P}(S)$-name for a function $\omega_1 \to \ON$ whose initial segments are in $V$, then $\dot{f}$ is forced to be in $V$.\cite{Krueger2007}
\end{enumerate}
\end{fact}

\subsection{The proof} Now we prove \autoref{secondarytheorem}. Fix $\lambda$ Mahlo. We can assume that $\Diamond_\lambda(\textup{Reg})$ holds, so let $\ell$ witness \autoref{mahlolaverfunction}.

Let $\mathbb{I}=\seq{\mathbb{I}_\alpha,\dot{\mathbb{J}}_\alpha}{\alpha<\lambda}$ be a countable-support iteration of length $\lambda$ such that if $\ell(\delta)$ is an $\mathbb{I}_\delta$-name for a proper forcing then $\Vdash_{\mathbb{I}_\delta} \textup{``} \dot{\mathbb{J}}_\delta = \ell(\delta)\textup{''}$ and otherwise $\dot{\mathbb{J}}_\delta$ is forced to be the trivial forcing.\footnote{See Abraham \cite{Abraham1983} and Cummings-Foreman \cite{Cummings-Foreman1998} for classical examples of forcings that use guessings functions in their definitions.} We will argue momentarily that we have $\Vdash_{\mathbb{I}} \textup{``}\aleph_1^{<\aleph_1} \le \aleph_2 =\lambda\textup{''}$, so we fix an $\I$-name $\underline{\dot{x}}$ of $[\aleph_2]^{<\aleph_1}$ in $V[\I]$ as well as a sequence of names $\seq{\dot{x}_\alpha}{\alpha<\aleph_2}$ that canonically represent the elements listed by $\underline{\dot{x}}$. Then let $\dot{\mathcal{C}}$ be an $\I$-name for $\mathcal{C}(\underline{\dot x})$. Let $G$ be $\I$-generic over $V$ and let $H$ be $\mathcal{C}:=\dot{\mathcal{C}}[G]$-generic over $V[G]$. Then the model in which the theorem is realized is $V[G][H]$.

Most of the desired properties of $V[G][H]$ follow easily. First of all, $V[G][H] \models \lambda = \aleph_2$: For all $\Theta<\lambda$ the forcing $\Col(\aleph_1,\Theta)$ appears in the iteration, $\mathbb{I}$ has the $\lambda$-chain condition because the iterands have size less than $<\lambda$, and $\mathbb{I}$ preserves $\aleph_1$ because it is proper. Then adjoining $H$ preserves $\aleph_2$ by the distributivity property noted above. The fact that $V[G][H] \models \neg \DSS(\mu^+)$ follows from \autoref{canonical-ds} given that the generic object added by $H$ is a club through the complement of the relevant stationary set. The main part of the work is to show that the approachability property fails.

\emph{Note:} If $\mathcal{M} \prec \mathcal{N}$ is rich and $\pi_{\mathcal{M}}$ is the Mostowski collapse relative to $\mathcal{M}$, we will typically denote $\pi_{\mathcal{M}}(a)$ as $\bar a$.


\begin{lemma}$V[G][H] \models \neg \AP(\aleph_2)$.\end{lemma}

\begin{proof} If $\AP(\aleph_2)$ holds then this is forced by some condition $z \in \I \ast \mathcal{C}$. Assuming this is the case, we can derive a contradiction.

\begin{claim}\label{liftinglemma} Let $\mathcal{M} \prec \mathcal{N}$ be a rich model chosen to witness \autoref{mahlolaverfunction} in the sense of having the properties that $\mathcal{M} \cap \lambda = \bar{\lambda}$, $z \in \mathcal{M}$, and $\ell(\bar{\lambda}) $ is an $\mathbb{I} \rest \bar{\lambda}$-name for
\[
\pi_\mathcal{M}(\dot{\mathcal{C}}(\dot{\underline{x}}) \ast \Add(\omega) \ast \P(Y) \ast \B(Y)).
\]
where $Y = ([\lambda]^\omega)^{V[\I \ast \dot{\mathcal{C}}(\dot{\underline{x}})]}$, and moreover $\P(Y)$ is defined with respect to the interpretation of $Y$ as a stationary set and $\B(Y)$ is defined with respect to the interpretation of $Y$ as a tree ordered by end-extension.

Suppose $\bar{G}_0 \ast \bar{H}_0$ is $\bar{\I} \ast \bar{\mathcal{C}}$-generic over $V$. Then there is a $G_0 \ast H_0$ which is $\I \ast \mathcal{C}(\underline{\dot{x}})$-generic over $V$ such that if $j:\bar{\mathcal{M}} \to \mathcal{M} \subset \mathcal{N}$ is the inverse of the Mostowski collapse, then there is a lift $j:\bar{\mathcal M}[\bar{G}_0][\bar{H}_0] \to \mathcal{N}[G_0][H_0]$ with the property that $G_0 \ast H_0$ is an $\aleph_1$-preserving extension over $V[\bar{G}_0][\bar{H}_0][\bar{K}_0][\bar{K}_1][\bar{K}_2]$ where $\bar{K}_0 \ast \bar{K}_1 \ast \bar{K}_2$ is $\Add(\omega) \ast \bar{\P}(Y) \ast \bar{\B}(Y)$-generic.

\end{claim}

\begin{proof}[Proof of Claim] We will lift the elementary embedding $j:\bar{\mathcal{M}} \to \mathcal{N}$ to $j:\bar{\mathcal{M}}[\bar{G}_0][\bar{H}_0] \to \mathcal{N}[G_0][H_0]$. We therefore fix the notation $\bar \lambda = \mathcal{M} \cap \lambda$, and we have an $\bar{\mathbb{M}}$-generic $\bar{G}_0$, so we let $\mathcal{C}=\dot{\mathcal{C}}(\underline{\dot{x}})[G_0]$.

To perform the lift, we need to show that we can absorb the generic $\bar{H}_0$. The first stage is for handling $G_0$. The forcing $\dot{\mathcal{C}}(\dot{\underline{x}}) \ast \Add(\omega) \ast \P(Y) \ast \B(Y)$ is an iteration of proper forcings and is therefore proper, and its image under $\pi_{\mathcal{M}}$ is proper for similar reasons. Hence, since it is also guessed, it is used in the iteration. Therefore $G_0$ takes the form $\bar{G}_0 \ast \bar{H}_0 \ast \bar{K}_0 \ast \bar{K}_1 \ast \bar{K}_2 \ast \bar{K}_3$ where $\bar{K}_3$ is just a remainder. The quotient preserves $\aleph_1$ since the whole forcing does.

To lift the embedding further, we use a master condition argument. Specifically, we want to show that $\cup \bar{H}_0 \cup \{\bar \lambda\}$ is a condition in $\mathcal{C}$. This follows because $\bar \lambda \notin S(\underline{x})$ as evaluated in $\mathcal{N}[G_0]$: Since $\bar{\mathcal{M}}^{<\bar{\lambda}} \subseteq \bar{\mathcal{M}}$ and $\I \rest \bar{\lambda}$ has the $\bar{\lambda}$-chain condition, the evaluation $\seq{x_\beta}{\beta<\bar{\lambda}}$ is equal to the countable subsets of $\bar{\lambda}$ in $\bar{\mathcal{M}}[\bar{G}_0]$. Therefore $P_\mu(\bar \lambda) \setminus \seq{x_\beta}{\beta<\bar \lambda}$ will be nonstationary because of the club added by $\P(Y)$. Hence we choose $H_0$ to be a generic containing $\cup \bar{H}_0 \cup \{\bar \lambda\}$.\end{proof}


Suppose then that $z \in \M \ast \dot{\mathcal{C}}(\underline x)$ forces that approachability holds. By the claim, there is an embedding $\bar{\mathcal{M}}[\bar{G}][\bar{H}] \to \mathcal{N}[G][H]$ such that $V[G \ast H]$ is an extension over $V[\bar{G}][\bar{H}][\bar{K}_0][\bar{K}_1][\bar{K}_2]$ that preserves $\aleph_1$ where $\bar{K}_0 \ast \bar{K}_1 \ast \bar{K}_2$ is generic for $\pi_{\mathcal{M}}(\Add(\omega) \ast \P(Y) \ast \B(Y))$. Since we are supposing that approachability holds, there is in $\mathcal{N}[G][H]$ a club $C \subseteq \aleph_2$ such that all of its points of cofinality $\aleph_1$ are approachable. By elementarity it follows that $\bar{\lambda} \in C$, so it is enough to show that $\bar{\lambda}$ cannot actually be an approachable point.

We need to show that $\bar{Y}$ does not have a cofinal branch. The set $\bar{Y}$ is already in $V[\bar{G}][\bar{H}]$, and it is an Aronszajn tree in this model because $V[\bar{G}][\bar{H}] \models \textup{``}\bar{\lambda} = \aleph_2\textup{''}$ and any branch through $\bar{Y}$ of length $\aleph_1$ would be a collapse of $\bar{\lambda}$. By the weak $\omega_1$-approximation property of $\Add(\omega) \ast \dot{\P}(S)$ (\autoref{some-facts}), $\bar{Y}$ is still an Aronszajn tree in $V[\bar{G}][\bar{H}][\bar{K}_0][\bar{K}_1]$ because no new cofinal branches are added. Moreover it has cardinality $\aleph_1$ in that model. (Note that $\bar{Y}$ may have uncountable levels.) The forcing $\bar{K}_2$ adds a specializing function, therefore it remains an Aronszajn tree in any $\aleph_1$-preserving extension, so in particular this is true for $V[G][H]$.\end{proof}

\begin{remark} This master condition argument can also be used to show that $\mathcal{C}(\underline x)$ is distributive over $V[\I]$.\end{remark}

Now we are finished with the proof of \autoref{secondarytheorem}.

\subsection*{Acknowledgement} Many thanks to Hannes Jakob for finding errors in the original version of this paper, and to the anonymous referees for their further corrections and improvements.

\section{Further directions}

We propose some other considerations along the lines of the question: Why did we have to do more work to get \autoref{superconsecutivetheorem} after obtaining \autoref{consecutivetheorem}? Or rather, is the assumption $2^\mu = \mu^+$ necessary for \autoref{dss-equivalence}?

\begin{question} Is it consistent for $\mu$ regular that exactly one of $\DSS(\mu^+)$ and ``internally club and internally unbounded are distinct for $[H(\mu^+)]^\mu$'' holds?\footnote{This question was answered by Jakob after the previous version of this paper was released \cite{Jakob2024}.}\end{question}


On a similar note, the assumption that $2^\mu=|H(\mu^+)|$ is also used in a folklore result that assuming $2^\mu=\mu^+$, the distinction between internally unbounded and internally approachable for $[\mu^+]^\mu$ requires a Mahlo cardinal.

\begin{question} What is the exact equiconsistency strength of the separation of internally approachable and internally unbounded for $[H(\mu^+)]^\mu$ for regular $\mu$?\end{question}

\bibliography{bibliography}
\bibliographystyle{plain}

\end{document}